\documentclass[11pt]{amsart}
\usepackage{amscd}
\usepackage{amsmath}
\usepackage{amsthm}
\usepackage{amsfonts}
\usepackage{amssymb}
\usepackage{color}

\numberwithin{equation}{section}

\newtheorem{pro}{Proposition}[section]

\newtheorem{lemma}[pro]{Lemma}
\newtheorem{theorem}[pro]{Theorem}
\newtheorem{remark}[pro]{Remark}
\newtheorem{corollary}[pro]{Corollary}
\newtheorem*{corollary*}{Corollary}
\newtheorem*{thm*}{Theorem}
\newtheorem*{thmA}{Theorem A}

\newtheorem*{corollaryB}{Corollary B}
\newtheorem*{corollaryC}{Corollary C}

\newtheorem*{thmE}{Theorem E}

{\theoremstyle{definition}
\newtheorem{definition}[pro]{Definition}

 }

{\theoremstyle{remark}
}

\newcommand{\al}{\alpha}
\newcommand{\be}{\beta}
\newcommand{\ga}{\gamma}

\newcommand{\se}{\theta}

\newcommand{\lam}{\lambda}

\newcommand{\ra}{\rightarrow}
\newcommand{\lra}{\longrightarrow}

\newcommand{\CC}{{\mathbb C}}

\newcommand{\RR}{{\mathbb R}}
\newcommand{\QQ}{{\mathbb Q}}
\newcommand{\ZZ}{{\mathbb Z}}

\newcommand{\fA}{{\mathsf{A}}}

\newcommand{\cf}{{cf$.$\,}}

\newcommand{\ie}{{i.e$.$\,}}

\begin{document}

\pagestyle{myheadings} \thispagestyle{empty} \setcounter{page}{1}

\title[]
{Torus Actions and the Halperin-Carlsson Conjecture}

\author[]{Yoshinobu Kamishima}
\address{Department of Mathematics\\
Tokyo Metropolitan University\\
Minami-Ohsawa 1-1,Hachioji, Tokyo 192-0397, JAPAN}
\email{kami@tmu.ac.jp}
\author[]{Mayumi Nakayama}
\address{Department of Mathematics and
Information of Sciences\\
Tokyo Metropolitan
University\\
Minami-Ohsawa 1-1, Hachioji, Tokyo 192-0397, JAPAN}
\email{nakayama-mayumi@ed.tmu.ac.jp}

\date{\today}






\begin{abstract}
We give an affirmative answer to the Halperin-Carlsson conjecture
for the homologically injective torus actions on closed manifolds.\\\
This class contains \emph{holomorphic torus actions on compact
K\"ahler manifolds}, \emph{torus actions on compact Riemannian flat
manifolds}.
\end{abstract}

\maketitle

\section{Introduction}
Recently the real Bott tower and its generalization have been
studied by several people (\cite{CMS}, \cite{LM}, \cite{MP},
\cite{MN},\cite{KN}). A real Bott manifold is originally defined to
be the set of real points in the Bott manifold \cite{GK}.  Among
several characterizations by
 group actions,
the Halperin-Carlsson conjecture is true for real Bott manifolds.
The Halperin-Carlsson torus conjecture says that if there is an
almost free torus action $T^k$ on a closed $n$-manifold $M$, the
following inequality holds:
\begin{equation}\label{HCC}
2^k\leq \mathop{\sum}_{j=0}^n b_j.
\end{equation}Here $b_j={\rm rank}\, H_j(M;\ZZ)$ is the j-th Betti number
of $M$. See \cite{PU} for details and the references therein, see
also \cite{HA}.

Another characterization is that a real Bott manifold $M$ is  a
euclidean space form (Riemannian flat manifold) admitting a torus
action $T^k$ with $k={\rm rank }\, H_1(M)$. It is conceivable
whether the \emph{Halperin-Carlsson conjecture} holds for compact
euclidean space forms more generally.

By this motivation we revisit the classical results of the Calabi
construction of euclidean space forms with nonzero $b_1$ \cite{CALB}
and the Conner-Raymond's injective torus actions \cite{C-R1}. In
this direction, we shall introduce \emph{injective-splitting action}
of a torus $T^k$ on closed manifolds more generally. Our purpose of
this paper is to prove the Halperin-Carlsson conjecture for such
torus actions affirmatively.

Let $T^k$ be a $k$-dimensional torus $(k\geq 1)$. Given an
\emph{effective} $T^k$-action on a closed manifold $M$,
 the orbit map at $x\in M$ is defined to be
 ${\rm ev}(t)=tx$ $({}^\forall\, t\in T^k)$. Put
 $\pi_1(T^k)=H_1(T^k;\ZZ)=\ZZ^k$ and $\pi_1(M)=\pi$.
The map ${\rm ev}$ induces a homomorphism $\displaystyle {\rm
ev}_\#: \ZZ^k\ra \pi$ and $\displaystyle {\rm ev}_*: \ZZ^k\ra
H_1(M;\ZZ)$ respectively. According to the definition of
Conner-Raymond \cite{C-R1}, if ${\rm ev}_\#$ is \emph{injective},
the action $(T^k,M)$ is said to be \emph{injective}. ( Refer to
\cite[Theorem 2.4.2, also Subsection 11.1]{LR} for the definition to
be independent of the choice of
 the base  point $x\in M$.)
 Classically it is known that $\displaystyle {\rm
ev}_\#$ is injective for closed \emph{aspherical} manifolds
\cite{C-R}. On the other hand, if $\displaystyle {\rm ev}_*:
\ZZ^k\ra H_1(M;\ZZ)$, 
the $T^k$-action is said to be \emph{homologically injective}
\cite{C-R1}. We have shown

\begin{thmA}
If $T^k$ is a homologically injective action on a closed
$n$-manifold $M$, then
\begin{equation}
{}_kC_j\leq b_j.
\end{equation}
In particular the Halperin-Carlsson conjecture is true.
\end{thmA}

To prove Theorem A we revisit the Conner and Raymond's work
\cite{C-R1}. Then this is a consequence from
\emph{injective-splitting
 torus actions} more generally. See Theorem
\ref{splitfinite}. We verify that the following actions are in fact
injective-splitting torus actions.

\begin{corollaryB}
Every effective $T^k$-action on a compact $n$-dimensional euclidean
space form $M$ is homologically injective. Thus $\displaystyle
{}_kC_j\leq b_j$, the Halperin-Carlsson conjecture \eqref{HCC}
holds.
\end{corollaryB}

We obtain a characterization of \emph{holomorphic} torus actions
originally observed by Carrell \cite{CA}.
\begin{corollaryC}
Every holomorphic action of the complex torus $T^k_\CC$ on a
 compact K\"ahler manifold is homologically injective. In particular,
 $\displaystyle {}_{2k}C_j\leq b_j$, the
Halperin-Carlsson conjecture holds.
\end{corollaryC}

In Section 2, we introduce \emph{injective-splitting actions} on
closed manifolds and prove our main theorem \ref{splitfinite}. Using
this theorem, we show Corollaries B and C. In Section 3 we shall
give a proof concerning the existence of torus actions common to
both the Calabi's theorem and the Conner-Raymond's theorem as our
motivation (\cf Theorem \ref{reprove}).
\begin{thmE} A compact $n$-dimensional euclidean space form $M$
admits a homologically injective action of $T^k$ with $k={\rm
rank}\, H_1(M)$ in which ${\rm rank}\, C(\pi)={\rm rank}\, H_1(M)$.
\end{thmE}

\section{Injective torus actions}\label{split}
Suppose $(T^k,M)$ is an \emph{injective action} on a closed manifold
$M$. Let $\tilde M$ be the universal covering space of $M$ and
 denote $N_{{\rm Diff}(\tilde M)}(\pi)$ the normalizer of $\pi$ in
${\rm Diff}(\tilde M)$. The conjugation homomorphism $\displaystyle
\mu: N_{{\rm Diff}(\tilde M)}(\pi)\ra{\rm Aut}(\pi)$ defined by
$\mu(\tilde f)(\ga)=\tilde f\circ\ga\circ \tilde f^{-1}$
$({}^\forall\,\ga\in \pi)$ induces a homomorphism $\varphi$ which
has a commutative diagram:
\begin{equation}\label{lift1}
\begin{CD}
 1@. 1@. 1 \\
 @VVV @VVV @VVV\\
 C(\pi)@>>> \pi @>\mu>> {\rm Inn}(\pi)\\
  @VVV  @VVV @VVV\\
C_{{\rm Diff}(\tilde M)}(\pi)@>>> N_{{\rm Diff}(\tilde M)}(\pi)
@>\mu>> {\rm Aut}(\pi)\\
 @V\nu VV @V\nu VV @VVV\\
{\rm Diff}(M)^0\leq {\rm ker}\,\varphi@>>>{\rm Diff}(M)
@>\varphi>> {\rm Out}(\pi)\\
 @VVV @VVV @VVV\\
 1@. 1@. 1 \\
\end{CD}
\end{equation}(Compare \cite{LR1}.) Here $C_{{\rm Diff}(\tilde M)}(\pi)$
is the centralizer of $\pi$ in ${\rm Diff}(\tilde M)$. As $T^k\leq
{\rm Diff}(M)^0$, we have a lift $\tilde T^k\leq C_{{\rm
Diff}(\tilde M)}(\pi)$ to $\tilde M$.  Note  $\tilde T^k=\RR^k$. For
this, suppose  some $S^1\leq T^k$ lifts to $S^1$ (but not $\RR$) on
$\tilde M$. If $p:\tilde M\ra M$ is the covering map which is
equivariant; $p(tx)=t^mp(x)$ $(t\in S^1)$ for some $m\in \ZZ$,
chasing a commutative diagram
\begin{equation*}\label{lift2}
\begin{CD}
\pi_1(S^1) @>{\rm ev}_\#>> \pi_1(\tilde M)=1\\
@V{m\cdot} VV  @Vp_\#VV \\
\pi_1(S^1)@>{\rm ev}_\#>> \pi_1(M), \end{CD}
\end{equation*}
it follows ${\rm ev}_\#(m\ZZ)=1$. This contradicts the injectivity
of $S^1\leq T^k$.  We have a lift of groups from \eqref{lift1}:
\begin{equation}\label{lift2}
\begin{CD}
C(\pi)@>>> C_{{\rm Diff}(\tilde M)}(\pi)@>>>{\rm
Diff}(M)^0\leq {\rm ker}\, \varphi\\
@AAA  @AAA @AAA  \\
\ZZ^{k} @>>>\RR^k  @>>> T^{k}.
\end{CD}
\end{equation}
Since $\ZZ^k\leq C(\pi)$, letting $Q=\pi/\ZZ^k$, there is a central
group extension:
\begin{equation}\label{piinjective}
1\ra \ZZ^k\ra \pi\lra Q\ra 1. \end{equation}

 Now $\RR^k$ acts properly and freely on $\tilde M$ such that
$\tilde M=\RR^k\times W$ where $W=\tilde M/\RR^k$ is a simply
connected smooth manifold.
 The group extension \eqref{piinjective} represents a $2$-cocycle
$f$ in $H^2(Q;\ZZ^k)$ in which $\pi$ is viewed as the product
$\ZZ^k\times Q$ with group law:
\[ (n,\al)(m,\be)=(n+m+f(\al,\be),\al\be).\]
Let $Map(W,\RR^k)$ (respectively $Map(W,T^k)$) be the set of smooth
maps of $W$ into $\RR^k$ (respectively $T^k$) endowed with a
$Q$-module structure in which there is an exact sequence of
$Q$-modules \cite{C-R}:
\[
1\ra \ZZ^k\ra Map(W,\RR^k)\stackrel{\exp}\lra Map(W,T^k)\ra 1.\] As
$Q$ acts properly discontinuously on $W$ with compact quotient, it
follows from \cite[Lemma 8.5]{C-R} (also \cite{LR}):
\begin{equation}\label{vanishingR}
H^i(Q;Map(W,\RR^k))=0\, \ (i\geq 1)
\end{equation}so that the connected homomorphism
$\displaystyle \delta: H^1(Q; Map(W,T^k))\ra H^2(Q;\ZZ^k)$ is an
isomorphism. From this, there exists a map $\chi:Q\ra Map(W,\RR^k)$
such that $\delta^1\chi=f$. Then the $\pi$-action on $\tilde M$ can
be described as
\begin{equation}\label{product1}\begin{split}
(n,\al)(x,w)&=(n+x+\chi(\al)(\al w),\al w)\\
&\ \ \ \ ({}^\forall (n,\al)\in \pi, {}^\forall (x,w)\in \RR^k\times
W).
\end{split}
\end{equation}The action of $\pi$ may depend on the choice of
$\chi'$ such that $\delta^1\chi'=f$. However, the vanishing of
\eqref{vanishingR} shows
\begin{pro}\label{uniqueness}
 Such $\pi$-actions are equivalent to each other.
\end{pro}

Let $(T^k,M)$ be an injective $T^k$-action on a closed manifold $M$
which induces a central group extension \eqref{piinjective} as
above.

\begin{definition}\label{insp}
A $T^k$-action is said to be \emph{injective-splitting} if  there
exists a finite index normal subgroup $Q'$ of $Q$ such that the
induced extension splits;
\[\pi'=\ZZ^k\times Q'.\]
\end{definition}

 Here is a key result concerning
\emph{injective-splitting torus actions}.

\begin{theorem}\label{splitfinite}
Suppose that a closed manifold $M$ admits an injective-splitting
$T^k$-action. Then the following hold.
\begin{itemize}
\item[(i)]
$\displaystyle {}_kC_j\leq b_j$.
 In particular the Halperin-Carlsson conjecture is true.
\item[(ii)]  The $T^k$-action
is homologically injective.
\end{itemize}
\end{theorem}

\begin{proof}
\noindent {\bf Algebraic part.}\, \eqref{piinjective} induces a
commutative diagram:
\begin{equation}\label{split-extensionl}
\begin{CD}
1@>>> \ZZ^k@>>> \pi@>>>Q@>>> 1\\
@. ||@. @AA\iota A @AA\iota' A \\
1@>>> \ZZ^k@>>> \pi'@>>>Q'@>>> 1\\
\end{CD}\end{equation}Here $Q/Q'$ is a finite group
by Definition \ref{insp}. For the cocycle $f$ representing the upper
group extension, it follows $\iota'^*[f]=0\in H^2(Q';\ZZ^k)$ by the
hypothesis. We may assume
\begin{equation}\label{fzeroQ}
f|_{Q'}=0.
\end{equation}
On the other hand, if $\tau:H^2(Q';\ZZ^k)\ra H^2(Q;\ZZ^k)$ is the
transfer homomorphism (\cf \cite{BRO}, \cite{BRE}), then
$\tau\circ\iota'^*=|Q:Q'| : H^2(Q;\ZZ^k)\ra H^2(Q;\ZZ^k)$ so that
 $[f]$
is a torsion in $H^2(Q;\ZZ^k)$. There exists an integer $\ell$ such
that $\ell\cdot f=\delta^{1}\tilde\lam$ for some function
$\tilde\lam:Q\ra \ZZ^k$. Put $\displaystyle\lam=\frac
{\tilde\lam}{\ell}: Q\ra \RR^k$. Then
\begin{equation}\label{ffzeroQ}
f=\delta^{1}\lam.
\end{equation}\eqref{fzeroQ} shows 
$\displaystyle[\lam|_{Q'}]\in H^1(Q;\RR^k)$.  Viewed $\RR^k\leq
Map(W,\RR^k)$ as constant maps, $[\lam|_{Q'}]\in
H^1(Q;Map(W,\RR^k))=0$ by \eqref{vanishingR}. So there is an element
$h\in Map(W,\RR^k)$ such that $\lam|_{Q'}=\delta^0 h$. The equality
$\lam(\al')=\delta^0 h(\al')(w)$ $({}^{\forall}\,\al'\in
Q',{}^{\forall}\, w\in W)$ implies
\begin{equation}\label{equivh}
h(w)=h(\al' w)+\lam(\al').
\end{equation}

\noindent{\bf Geometric part.}\,
 Noting Proposition \ref{uniqueness},
 the $\pi$-action \eqref{product1} on $\tilde M$ is equivalent with
 \begin{equation}\label{product2}
(n,\al)(x,w)=(n+x+\lam(\al),\al w)\ \ ({}^\forall\, (x,w)\in
\RR^k\times W).
\end{equation}

Recall that $\pi$ has the splitting subgroup $\pi'=\ZZ^k\times Q'$.
Obviously we have the product action of $\ZZ^k\times Q'$ on
$\RR^k\times W$ such that $\RR^k\times W/\ZZ^k\times Q'=T^k\times
W/Q'$. Define a diffeomorphism $\tilde G:\RR^k\times W\ra
\RR^k\times W$ to be $\displaystyle \tilde G(x,w)=(x+h(w),w)$. Using
\eqref{equivh}, it is easy to check that $\tilde G:
(\pi',\RR^k\times W)\ra (\ZZ^k\times Q', \RR^k\times W)$ is an
equivariant diffeomorphism with respect to the action
\eqref{product2} and the product action. Putting $\displaystyle
\RR^k\times W/\pi'=T^k\mathop{\times}_{Q'} W$ as a quotient space,
$\tilde G$ induces a diffeomorphism $\displaystyle G:
T^k\mathop{\times}_{Q'} W\ra T^k\times W/Q'$. Let $\displaystyle q:
T^k\times W \ra T^k\mathop{\times}_{Q'} W$ be the covering map
($q(t,w)=[t,w]$). Then
\begin{equation}\label{quotinetdiff1}
G\circ q(t,w)=G([t,w])=(t\exp2\pi\mathbf{i}h(w),[w]).
\end{equation}
Noting \eqref{product2}, $\pi$ induces an action of $Q$ on $\tilde
M/\ZZ^k=T^k\times W$ such that
\begin{equation}\label{actionFF}
\al(t,w)=(t\exp 2\pi\mathbf{i}\lam(\al),\al w)\, \ ({}^\forall\,
\al\in Q).
\end{equation}
$F=Q/Q'$ has an induced action on $\displaystyle
T^k\mathop{\times}_{Q'} W$ by $\displaystyle \hat\al[t,w]=[t\exp
2\pi\mathbf{i}\lam(\al),\al w]$ $({}^\forall\, \hat \al\in F)$ which
gives rise to a covering map:
\begin{equation}\label{covering}
F\ra T^k\mathop{\times}_{Q'} W \stackrel{\nu}\lra
T^k\mathop{\times}_{Q} W=M.
\end{equation}

For any $\al\in Q$, consider the commutative diagram:
\begin{equation}\label{trivailind}
\begin{CD}
H_j(T^k\times W)@>\al_*>> H_j(T^k\times W)\\
@VV{q_*}V @VV{q_*}V \\
H_j(T^k\mathop{\times}_{Q'} W)@>\hat\al_*>>
H_j(T^k\mathop{\times}_{Q'} W)
\end{CD}\end{equation}
in which $H_j(T^k)\otimes H_0(W)\leq H_j(T^k\times W)$. By the
formula \eqref{actionFF}, the $Q$-action on the $T^k$-summand is a
translation by $\exp 2\pi\mathbf{i}\lam(\al)\in T^k$ so the homology
action $\al_*$ on $H_j(T^k)\otimes H_0(W)$ is trivial.
 If $\displaystyle H_j(T^k\mathop{\times}_{Q'} W)^F$
denotes the subgroup left fixed under the homology action for every
element $\hat \al\in F$, it follows
\begin{equation}\label{fixF}
q_*(H_j(T^k)\otimes H_0(W))\leq H_j(T^k\mathop{\times}_{Q'} W)^F.
\end{equation}
Using the transfer homomorphism (\cite[2.4 Theorem, III]{BRE},
\cite{BRO}), $\nu$ of\eqref{covering} induces an isomorphism:
$\displaystyle \nu_*:H_j(T^k\mathop{\times}_{Q'} W;\QQ)^F\lra
H_j(M;\QQ)$. In particular, $\displaystyle
\nu_*:q_*(H_j(T^k;\QQ)\otimes H_0(W;\QQ))\ra H_j(M;\QQ)$ is
injective.

On the other hand, let $\displaystyle q':W \ra W/Q'$ be the
projection $q'(w)=[w]$. Define
 a homotopy $\Psi_\se:T^k\times W\ra T^k\times W/Q'$ $(\se\in[0,1])$ to
 be $$\Psi_\se(t,w)=(t\exp2\pi\mathbf{i}(\se\cdot h(w)),[w]).$$
Then $\Psi_0={\rm id}\times q'\simeq G\circ q$ from
\eqref{quotinetdiff1}. As $G_*\circ q_*={\rm id}\times
q'_*:H_j(T^k;\QQ)\otimes H_0(W;\QQ)\ra H_j(T^k;\QQ)\otimes
H_0(W/Q';\QQ)$ is obviously isomorphic, it implies that
$\displaystyle q_*: H_j(T^k;\QQ)\otimes H_0(W;\QQ) \lra
H_j(T^k\mathop{\times}_{Q'} W;\QQ)$ is injective. If $\displaystyle
p=\nu\circ q: T^k\times W \ra M$ is the projection, then
$\displaystyle p_*: H_j(T^k;\QQ)\otimes H_0(W;\QQ)\lra H_j(M;\QQ)$
is injective.

As $p:T^k\times W \ra M$ is $T^k$-equivariant, letting $p(1,w)=x\in
M$, it follows $p(t,w)=tx={\rm ev}(t)$ $({}^\forall\, t\in T^k)$.
 Define an embedding
$\tilde{\rm ev}:T^k\ra T^k\times W$ to be $\tilde{\rm ev}(t)=(t,w)$.
Obviously $\tilde{\rm ev}_* :H_j(T^k;\QQ)\ra H_j(T^k;\QQ)\otimes
H_0(W;\QQ)$ is an isomorphism. Since $p\circ \tilde{\rm ev}(t)={\rm
ev}(t)$, chasing a commutative diagram:
\begin{equation}
\begin{CD}
H_j(T^k;\QQ)@>\tilde{\rm ev}_*>>  H_j(T^k;\QQ)\otimes H_0(W;\QQ)\\
 {\rm ev}_*\searrow @. \swarrow p_* \ \ \\
& \ \ H_j(M;\QQ) & \\
\end{CD}\end{equation}
$\displaystyle {\rm ev}_* :H_j(T^k;\QQ)\ra H_j(M;\QQ)$ is injective.
As $H_j(T^k;\ZZ)$ has no torsion,
 ${\rm ev}_* :H_j(T^k;\ZZ)\ra H_j(M;\ZZ)$ turns out to be injective.
 This proves (i) $\displaystyle {}_kC_j\leq b_j$.
In particular, ${\rm ev}_* :\ZZ^k\ra H_1(M;\ZZ)$ is injective for
$j=1$, \ie  the $T^k$-action is homologically injective by the
definition. This shows (ii). \end{proof}

Any homologically injective action is obviously injective. Theorem A
is obtained from the following corollary.
\begin{corollary}\label{CAl} If $T^k$ is a homologically
injective action on a closed manifold $M$, then $\displaystyle
{}_kC_j\leq b_j$. Thus the Halperin-Carlsson conjecture is true.
\end{corollary}

\begin{proof}The proof is essentially the same as \cite[2.2. Lemma]{C-R1}.
 Let $\displaystyle
1\ra \ZZ^k\ra \pi\lra Q\ra 1$ be the central group extension. As
${\rm ev}_*: H_1(T^k;\ZZ)=\ZZ^k\ra H_1(M;\ZZ)=\ZZ^{\ell}\oplus F$
 is injective, ${\rm ev}_*(\ZZ^k)\leq \ZZ^k$ such that ${\rm
ev}_*(\ZZ^k)\oplus \ZZ^{\ell-k}\leq \ZZ^\ell$. If $q:\pi\ra
H_1(M;\ZZ)$ is a canonical projection, then $\pi'=q^{-1}({\rm
ev}_*(\ZZ^k)\oplus\ZZ^{\ell-k}\oplus F)$ is a finite index normal
splitting subgroup of $\pi$.
\end{proof}

\begin{remark}\label{equiv}
As a consequence of this corollary and ${\rm (ii)}$ of Theorem
$\ref{splitfinite}$, \emph{injective-splitting} action is equivalent
with \emph{homologically injective} action.
\end{remark}

Let $(M,g)$ be a $2n$-dimensional K\"ahler manifold with K\"ahler
form $\Omega$.

\begin{corollary}[\cite{CA}]\label{Carrell}
Every \emph{holomorphic} action of a complex torus $T^k_\CC$ on a
compact K\"ahler manifold manifold $(M,\Omega)$ is homologically
injective. Thus the Halperin-Carlsson conjecture holds.
\end{corollary}

\begin{proof}Averaging the K\"ahler metric by $T^k_\CC$, we may assume
that $T^k_\CC$ acts as K\"ahler isometries on $M$. $T^k_\CC$ induces
the Killing vector fields $\xi_i, J\xi_i$ on $M$ $(i=1,\dots,k)$.
Note that each $\xi_i$ is a non-vanishing vector field by the
maximum principle. In fact, if $G$ is the connected component of the
stabilizer of $T^k_\CC$ at $x\in M$, then the action induces a
holomorphic representation $\rho:G\ra {\rm GL}(n,\CC)$. As $G$ is
compact, $\rho$ is trivial so that $G=\{1\}$.

Put $\theta_i=\iota_{\xi_i}\Omega$ and
$\theta_{i+k}=\iota_{J\xi_i}\Omega$ $(i=1,\dots,k)$. By the Cartan
formula, $d\Omega=0$ implies $d\theta_i=0$ $(i=1,\dots,2k)$.
 We have $2k$-number of
$1$-cocycles $[\theta_i]\in H^1(M;\RR)$.
As ${\rm ev}(t)=t\cdot x\in M$, it follows ${\rm
ev}_*((\xi_i)_1)=(\xi_i)_x$. Since ${\rm
ev}^*\theta_i((\xi_{k+i})_1)=
\Omega((\xi_i)_x,(J\xi_i)_x)=g(\xi_i,\xi_i)>0$ on $T^k_\CC$, ${\rm
ev}^*:H^1(M;\RR)\ra H^1(T^k_\CC;\RR)$ is surjective.  So ${\rm
ev}_*: H_1(T^k_\CC;\ZZ)\ra H_1(M;\ZZ)$ is injective.
\end{proof}

 For example, any
\emph{holomorphic} action of $T^k_\CC$ on a compact complex
euclidean space form is homologically injective. Recall that
 any effective $T^s$-action on a closed aspherical manifold is
injective. We prove Corollary B.

\begin{theorem}\label{flatHC}
 Any effective $T^s$-action on a compact
euclidean space form $M$ is homologically injective. Thus the
Halperin-Carlsson conjecture is true.
\end{theorem}

\begin{proof}Given a $T^s$-action
for some $s\geq 1$, there is a central group extension:
$\displaystyle 1\ra\ZZ^s\ra\pi=\pi_1(M)\lra Q\ra 1$. Since $\pi$ has
a unique maximal normal finite index abelian subgroup $\ZZ^n$ (\cf
\cite{WO}), it follows $\ZZ^s\leq \ZZ^n$. $Q$ has a finite index
subgroup $Q'=\ZZ^n/\ZZ^s\cong G\times \ZZ^{n-s}$ where $G$ is a
finite abelian group. The inclusion $\iota:\ZZ^{n-s}\ra Q'$ induces
a group (extension) $\fA$ of finite index in $\ZZ^n$. As $\fA$ is
isomorphic to $\ZZ^s\times \ZZ^{n-s}$, $\fA$ is a finite index
normal splitting subgroup of $\pi$. The $T^s$-action is
injective-splitting so apply Theorem \ref{splitfinite}.
\end{proof}

\section{Calabi construction and
torus actions} \label{concluding}

 In \cite[\S\, 7]{C-R1}, Conner and
Raymond have stated that the Calabi's theorem \cite{CALB} shows the
existence of a $T^k$-action with $k={\rm rank}\,H_1(M;\ZZ)>0$. We
agree the existence of such actions in view of the Calabi
construction. However when we look at a proof of the Calabi's
theorem (\cite[p.125]{WO}), it is not easy to find such
$T^k$-actions. In fact, let $\nu:\pi\ra \ZZ^k$ be the projection
onto the direct summand $\ZZ^k$ of $H_1(M;\ZZ)$. Then there is a
group extension $\displaystyle 1\ra \Gamma\ra \pi\ra \ZZ^k\ra 1$ in
which $\Gamma$ is the fundamental group of a euclidean space form
$M^{n-k}=\RR^{n-k}/\Gamma$. In general an element $\ga\in \pi$ has
the form
\[
\left(\left[\begin{array}{c}
a\\
b\\
\end{array}\right],  \left(\begin{array}{lr}
A& B\\
0& I
\end{array}\right)\right)\ \ (a\in \RR^{n-k},b\in \RR^k).  \]
The holonomy group $\displaystyle L(\pi)=\{\left(\begin{array}{lr}
A& B\\
0& I
\end{array}\right)\}$ does not necessarily leave the subspace
$0\times\RR^k$ invariant. (In particular, $\ZZ^n\cap (0\times\RR^k)$
is not necessarily uniform in $0\times\RR^k$.) So we have to find
another decomposition to get a $T^k$-action on $M$.

\begin{lemma}\label{Bieba}
Let $\pi$ be a Bieberbach group such that ${\rm rank}\,
\pi/[\pi,\pi]=k>0$. Then there exists a faithful representation
$\rho:\pi\ra {\rm E}(n)$ such that the euclidean space form
$\RR^n/\rho(\pi)$ admits an effective $T^k$-action.
\end{lemma}

\begin{proof}By the hypothesis, there is a group extension
$\displaystyle 1\ra \Gamma\ra \pi\stackrel{\nu}\lra \ZZ^k\ra 1$.
Since $\pi$ is a Bieberbach group, it admits a maximal normal finite
index abelian subgroup $\ZZ^n$. Put $\nu(\ZZ^n)=A$. Consider the
commutative diagram of the group extensions:
\begin{equation}\begin{CD}
1@>>> \Gamma @>\iota>>\pi@>\nu>> \ZZ^k@>>>1\\
@. @AAA @AAA @AAA \\
1@>>> \Gamma\cap \ZZ^n@>\iota>>\ZZ^n@>\nu>> A@>>>1.\\
\end{CD}\end{equation}
Since $\pi/\ZZ^n\stackrel{\hat\nu}\ra \ZZ^k/A$ is surjective, $A$ is
a free abelian subgroup of rank $k$. By the embedding $\displaystyle
\hat\iota: \Gamma/\Gamma\cap \ZZ^n\leq \pi/\ZZ^n$, $\Gamma\cap
\ZZ^n$ is a finite index subgroup of $\Gamma$. It follows easily
that $\Gamma\cap \ZZ^n$ is a maximal normal abelian subgroup of
$\Gamma$. We may put $\displaystyle \Gamma\cap \ZZ^n=\ZZ^{n-k}$ so
that $\ZZ^n=\ZZ^{n-k}\times A$. Putting $Q=\pi/\ZZ^{n-k}$ and
$F=\Gamma/\ZZ^{n-k}$ is a finite group, we have the group
extensions:
\begin{equation}
\label{Biextension0} \begin{CD}
1@>>> \ZZ^{n-k}@>i>> \pi@>\mu>>Q@>>> 1,\\
\end{CD}\end{equation}where
\begin{equation}\label{Biextension1}
\begin{CD}
1@>>> F@>\hat \iota>> Q @>\hat\nu>>\ZZ^k @>>> 1
\end{CD}\end{equation}is also a group extension.
Consider the commutative diagram of group extensions:
\begin{equation} \label{Biextension2}
\begin{CD}
1@>>> \ZZ^{n-k}@>\iota>> \pi@>\mu>>Q@>>> 1\\
 @.|| @. @AAA @AA\iota'A \\
1@>>> \ZZ^{n-k}@>\iota>> \ZZ^n@>\mu>>B@>>> 1\\
\end{CD}\end{equation}where we put $B=\mu(\ZZ^n)$.
As $\hat\nu(B)=\nu(\ZZ^n)=A$ from \eqref{Biextension0}, it follows
$\ZZ^n=\ZZ^{n-k}\times B$. Thus $\ZZ^n$ is a splitting subgroup of
$\pi$. Since \eqref{Biextension0} is not necessarily central, let
$\phi:Q\ra {\rm Aut}(\ZZ^{n-k})$ be the conjugation homomorphism. If
$[f]\in H^2_\phi(Q;\ZZ^k)$ is the representative cocycle of
\eqref{Biextension0}, then $\iota'^*[f]=0$ in $H^2(B;\ZZ^{n-k})$.
Then $[f]$ is a torsion because $\tau\circ\iota'^*=|Q/B| :
H^2_\phi(Q;\ZZ^k)\ra H^2_\phi(Q;\ZZ^k)$ still holds for the transfer
homomorphism $\tau:H^2(B;\ZZ^{n-k})\ra H^2_\phi(Q;\ZZ^{n-k})$.
(Compare \cite{BRO}.) Similarly as in the proof of Theorem
\ref{splitfinite} there is a function $\lam:Q\ra \RR^{n-k}$ such
that $f=\delta^{1}\lam$. Note from \eqref{ffzeroQ} that
\begin{equation}\label{integral} \ell\cdot\lam(Q)\leq
\ZZ^{n-k}.
\end{equation}

Let $\ZZ^k$ act on $\RR^k$ by translations and by
\eqref{Biextension1} $\hat\nu:Q\ra \ZZ^k\leq {\rm E}(k)$ defines a
properly discontinuous action of $Q$ on $\RR^k$;
$$
\al(w)=\hat\nu(\al)+w\ \,({}^\forall\,\al\in Q, {}^\forall\, w\in
\RR^k).$$ As in \eqref{product1} we have a properly discontinuous
action of $\pi$ on $\RR^n=\RR^{n-k}\times \RR^k$:
\begin{equation*}\label{piaction}
\begin{split}
(n,\al)\left[\begin{array}{c} x\\
w\end{array}\right]&=\left[\begin{array}{c}
n+\bar\phi(\al)(x)+\lam(\al)\\
\hat\nu(\al)+ w\end{array}\right]\\ &=
\left(\left[\begin{array}{c} n+\lam(\al)\\
\hat\nu(\al)\end{array}\right],\left(\begin{array}{lr} \bar\phi(\al) & \\
& I_k\end{array}\right)\right)\left[\begin{array}{c} x\\
w\end{array}\right].
\end{split}
\end{equation*}
Since $\phi|_B={\rm id}$ from \eqref{Biextension2}, the image
$\phi(Q)$ is finite in ${\rm Aut}(\ZZ^{n-k})$ which implies
$\phi(Q)\leq {\rm O}(n-k)$ up to conjugate. As $\pi$ is torsionfree
and acts properly discontinuously,
 we obtain a faithful
homomorphism $\rho:\pi \ra \rho(\pi)\leq{\rm E}(n)$ defined by
\begin{equation}\label{affineaction2}
\rho(n,\al)=\left(\left[\begin{array}{c} n+\lam(\al)\\
\hat\nu(\al)\end{array}\right],\left(\begin{array}{lr} \bar\phi(\al) & \\
& I_k\end{array}\right)\right).
\end{equation}
Therefore we have a compact euclidean space form $\RR^n/\rho(\pi)$.\\

We prove that $\RR^n/\rho(\pi)$ admits a $T^k$-action. Noting
\eqref{integral}, we define a subgroup of $\ZZ^n$ by
\begin{equation}\label{translationsub}
\tilde B=\{(-\ell\cdot \lam(\be),\ell\cdot\beta)\in \ZZ^n \, |\,
\be\in B\}. \end{equation} It is isomorphic to $B\cong \ZZ^k$. As
$\phi|_B={\rm id}$,
\begin{equation}\label{trans2}
\rho(-\ell\cdot \lam(\be),\ell\cdot\beta)=\left(\left[\begin{array}{c} 0\\
\ell\cdot\hat\nu(\be)\end{array}\right],I_n\right)\in 0\times \RR^k.
\end{equation}Thus $\rho(\tilde B)$ is a translation subgroup
with rank $k$:
\begin{equation}\label{cocompact}
 \rho(\tilde B)\leq (0\times \RR^k)\cap
\rho(\pi).
\end{equation}Since
$(0\times \RR^k)/\rho(\tilde B)$ is compact, so is $(0\times
\RR^k)/(0\times \RR^k)\cap \rho(\pi)$. We may put $\displaystyle
(0\times \RR^k)/(0\times \RR^k)\cap \rho(\pi)=T^k$. Moreover, from
\eqref {affineaction2} a calculation shows that
\begin{equation}
\rho(n,\al)\cdot
\left(\left[\begin{array}{c} 0\\
y\end{array}\right],I_n\right)\\
=\left(\left[\begin{array}{c} 0\\
y\end{array}\right],I_n\right) \cdot \rho(n,\al),
\end{equation}\ie each $y\in\RR^k$ centralizes $\rho(\pi)$;
\begin{equation}\label{cetergroup}
0\times \RR^k\leq C_{{\rm E}(n)}(\rho(\pi)).
\end{equation} Let ${\rm Isom}(\RR^k/\rho(\pi))^0$
denote the identity component of euclidean isometries of
$\RR^k/\rho(\pi)$. From \eqref{cetergroup} we have the following
covering groups (\cf \eqref{lift2}):
\begin{equation*}\label{pulbackcovering-m}
\begin{CD}
1@>>>C(\rho(\pi)) @>>> C_{{\rm E}(n)}(\rho(\pi))@>>> {\rm Isom}(\RR^k/\rho(\pi))^0 \\
@. @AAA  @AAA @AAA \\
1@>>> (0\times \RR^{k})\cap \rho(\pi) @>>> (0\times \RR^{k}) @>>> T^{k}\\
\end{CD}
\end{equation*}
Hence $\RR^k/\rho(\pi)$ admits a $T^k$-action.
\end{proof}

When the center $C(\pi)$ of $\pi=\pi_1(M)$ for a closed manifold $M$
is finitely generated, ${\rm rank}\, C(\pi)$ denotes the rank of a
free abelian subgroup. Recall from \cite{LR2} that an effective
$T^s$-action on $M$ is said to be \emph{maximal} if $s={\rm rank}\,
C(\pi)$. If $M$ is a closed aspherical manifold admitting an
effective $T^s$-action, then $s\leq {\rm rank}\, C(\pi)$ (\cf
\cite{C-R}). Let $M$ be a euclidean space form $M=\RR^n/\pi$. If
$C(\pi)$ has rank $k$, then it is easy to see that $M$ admits a
$T^k$-action. In particular, ${\rm rank}\, C(\pi)\leq {\rm rank}\,
H_1(M;\ZZ)$ by Theorem \ref{splitfinite}.\\

The Bieberbach theorem implies that $M$ is affinely diffeomorphic to
$\RR^n/\rho(\pi)$. Combined Lemma \ref{Bieba} with Theorem
\ref{flatHC}, we obtain

\begin{theorem}\label{reprove}
 Let $M$ be a compact $n$-dimensional euclidean space form with
${\rm rank}\, H_1(M;\ZZ)=k>0$. Then $M$ admits a homologically
injective $T^k$-action. In particular, ${\rm rank}\, C(\pi)={\rm
rank}\, H_1(M;\ZZ)$.
\end{theorem}

\noindent{{\bf{Acknowledgement}.}\,{ We thank Professor M. Masuda
who called our attention to the Halperin-Carlsson conjecture.}

\bibliographystyle{amsplain}

\end{document}